\documentclass[reqno]{amsart}

\usepackage{amsmath}
\usepackage{amssymb}

\numberwithin{equation}{section}

\newtheorem{Theorem}{Theorem}[section]
\newtheorem{Lemma}[Theorem]{Lemma}
\newtheorem{Proposition}[Theorem]{Proposition}

\theoremstyle{definition}

\newtheorem{Remark}[Theorem]{Remark}
\newtheorem{Example}[Theorem]{Example}

\newcommand{\thref}[1]{Theorem \ref{#1}}
\newcommand{\leref}[1]{Lemma \ref{#1}}
\newcommand{\prref}[1]{Proposition \ref{#1}}
\newcommand{\reref}[1]{Remark \ref{#1}}

\newcommand{\seref}[1]{Section \ref{#1}}

\newcommand{\Cset}{\mathbb{C}}
\newcommand{\Rset}{\mathbb{R}}
\newcommand{\Zset}{\mathbb{Z}}
\newcommand{\Nset}{\mathbb{N}}

\newcommand{\Id}{\mathrm{Id}}
\newcommand{\Span}{\mathrm{span}}
\newcommand{\Wr}{\mathrm{Wr}}

\newcommand{\al}{{\alpha}}
\newcommand{\be}{{\beta}}
\newcommand{\la}{{\lambda}}
\newcommand{\ka}{{\kappa}}
\newcommand{\de}{{\delta}}
\newcommand{\De}{{\Delta}}

\newcommand{\taub}{\bar{\tau}}
\newcommand{\ep}{{\epsilon}}
\newcommand{\var}{\varepsilon}

\newcommand{\dms}{d\mu_{\mathrm{d}}(s)}

\newcommand{\fh}{\hat{f}}

\newcommand{\ph}{\hat{p}}

\newcommand{\dx}{{\frac{d}{dx}}}
\newcommand{\fp}{{\mathfrak p}}
\newcommand{\fq}{{\mathfrak q}}

\newcommand{\cM}{{\mathcal M}}

\newcommand{\cA}{{\mathcal A}}
\newcommand{\cBh}{\hat{{\mathcal B}}}
\newcommand{\cBb}{\bar{{\mathcal B}}}

\newcommand{\fAz}{{\mathfrak A}_z}
\newcommand{\fAn}{{\mathfrak A}_n}

\newcommand{\fDz}{{\mathfrak D}_z}
\newcommand{\fDn}{{\mathfrak D}_n}

\newcommand{\fb}{\bar{f}}

\newcommand{\ga}{{\gamma}}

\newcommand{\xia}{\xi^a}
\newcommand{\xib}{\xi^b}
\newcommand{\xic}{\xi^c}
\newcommand{\xid}{\xi^d}

\newcommand{\cL}{{\mathcal L}}
\newcommand{\cLn}{{\mathcal L}_n}
\newcommand{\cLga}{{\mathcal L}_{\ga}}

\newcommand{\cBz}{{\mathcal B}_z}
\newcommand{\cQ}{{\mathcal Q}}

\newcommand{\cLh}{{\hat{\mathcal L}}}

\newcommand{\fpt}[7]{{}_4\phi_3\left[\begin{matrix} #1 , #2, #3, #4 \\
#5, #6, #7 \end{matrix}\,; q,q\right]}

\begin{document}
\title{Bispectral extensions of the Askey-Wilson polynomials}

\date{May 27, 2013}

\author[P.~Iliev]{Plamen~Iliev}
\address{School of Mathematics, Georgia Institute of Technology,
Atlanta, GA 30332--0160, USA}
\email{iliev@math.gatech.edu}
\thanks{The author is partially supported by a grant from the Simons Foundation.}

\begin{abstract}
Following the pioneering work of Duistermaat and Gr\"unbaum, we call a family $\{p_n(x)\}_{n=0}^{\infty}$ of polynomials bispectral, if the polynomials are simultaneously eigenfunctions of two commutative algebras of operators: one consisting of difference operators acting on the degree index $n$, and another one of operators acting on the variable $x$. The goal of the present paper is to construct and parametrize bispectral extensions of the Askey-Wilson polynomials, where the second algebra consists of $q$-difference operators. In particular, we describe explicitly measures on the real line for which the corresponding orthogonal polynomials satisfy (higher-order) $q$-difference equations extending all known families of orthogonal polynomials satisfying $q$-difference, difference or differential equations in $x$. \end{abstract}

\maketitle


\section{Introduction}\label{se1}
Orthogonal polynomials which are eigenfunctions of a differential operator have a long history. In 1929, Bochner \cite{Bo} proved that, up to an affine change of variables, the classical orthogonal polynomials are the only ones that satisfy a differential equation of the form
\begin{equation}\label{1.1}
B\left(x,\dx\right)p_n(x)=\lambda_n p_n(x),\quad n\in\Nset_0,
\end{equation}
where $B(x,\dx)=\sum_{j=0}^{2}b_j(x)\frac{d^j}{dx^j}$ is a second-order differential operator with coefficients $b_j(x)$ independent of the degree index $n$. 
In 1938,  Krall \cite{Kr1} posed the general problem to construct and classify all families of orthogonal polynomials  which are eigenfunctions of a differential operator $B\left(x,\dx\right)$ of arbitrary order, which is independent of the degree index $n$. He showed that the order of the operator must be even and solved it completely for operators of order four \cite{Kr2}. During the next 60 years, many different examples of orthogonal polynomials satisfying higher-order differential equations were constructed, see for instance \cite{KK,Ko,L} and the references therein.

The pioneering work of Duistermaat and Gr\"unbaum \cite{DG} brought new bispectral techniques in this old problem. Roughly speaking, the bispectral problem concerns the construction and the characterization of functions depending on two variables which satisfy  simultaneously spectral equations in both variables. For the classical or Krall polynomials, the bispectral property can be explained as follows. Recall that for any family of orthogonal polynomials, the operator multiplication by $x$ can be represented by a three-term recurrence relation:
\begin{equation}\label{1.2}
xp_n(x)=A_{n} p_{n+1}(x)+B_np_n(x)+C_np_{n-1}(x),
\end{equation}
where the coefficients $A_n$, $B_n$, $C_n$ are independent of $x$. Looking now at equation \eqref{1.1} we see that the polynomials are eigenfunctions of a differential operator acting on the variable $x$, while equation \eqref{1.2} shows that they are also eigenfunctions of a second-order difference operator acting on the degree index $n$. 

Although Duistermaat and Gr\"unbaum considered differential operators in both equations,  the connection between Krall's problem and the bispectral problem, suggested that the tools used in \cite{DG} can be translated within the context of orthogonal polynomials. And indeed, as shown in \cite{GH1} the Krall polynomials can be obtained from special instances of the classical orthogonal polynomials by applying the so-called Darboux transformation, one of the basic tools in the theory of solitons \cite{AM,MS}. Soon after that, the general bispectral techniques developed in \cite{BHY} led to the construction of far reaching extensions of the Laguerre polynomials \cite{GHH} and Jacobi polynomials \cite{GY} which are eigenfunctions of higher-order differential operators.  Things were taken even further in \cite{HI} where  extensions of the Askey-Wilson polynomials which are eigenfunctions of $q$-difference operators were constructed, containing as $q\rightarrow 1$ the previous results. 

One of the interesting questions left open in the papers  \cite{GHH,GY,HI} was to characterize explicitly the algebra of all possible differential or $q$-difference operators (and in particular, the operator of minimal order). The methods used there could only guarantee the existence of an operator of every sufficiently large order leading in practice to higher-order operators even in simple examples. Recently \cite{I2,I3} yet another soliton method was developed to establish the bispectrality of the generalized Jacobi and Laguerre polynomials.  One of the advantages of this approach is that it provides a more detailed information about the commutative algebras of differential operators which was crucial for multivariate extensions \cite{I2}. Moreover, it naturally leads to the construction of lower order differential operators, thus establishing conjectures concerning the characterization of the commutative algebra of all possible differential operators diagonalized by a specific family of orthogonal polynomials \cite{I3}. The roots of this method go back to the work of Reach \cite{R}.  Although its role in the Krall's problems was revealed very recently, we note that the method was adapted and systematically used to prove the bispectrality of specific rational solutions of the Kadomtsev-Petviashvili hierarchy \cite{Lib}, its discrete versions and $q$-deformations \cite{I1,I4}. 

The goal of the present paper is to extend the techniques in \cite{I2,I3} to the Askey-Wilson level, thus providing a uniform construction and parametrization of all known examples of bispectral polynomials which are eigenfunctions of two commutative algebras of operators: one consisting of difference operators acting on the degree index $n$, and another one of $q$-difference operators acting on the variable $x$. 
Since the Askey-Wilson polynomials are at the very top of the $q$-Askey scheme, the statements in the present paper lead to a variety of interesting special or limiting subfamilies of orthogonal polynomials which are eigenfunctions of $q$-difference, difference or differential operators, containing as special or limiting cases all known examples. Producing an explicit table with all such examples similar to the reductions of the Askey-Wilson polynomials \cite{KLS} is a challenging task, far beyond the scope of this paper. On contrary, we have tried to make the presentation compact and self-contained, emphasizing the main new ideas and constructions.

The paper is organized as follows. In the next section, we fix the notation and collect the main properties of the Askey-Wilson polynomials needed in the paper. In \seref{se3} we construct the extensions $\{\ph_n(x)\}_{n=0}^{\infty}$ of the Askey-Wilson polynomials and formulate the main result - their bispectrality. In \seref{se4} we establish the recurrence relations (the spectral equations in the degree index $n$). 
\seref{se5} is the heart of the paper where we construct the commutative algebra of $q$-difference operators in $z$, with $x=\frac{1}{2}(z+\frac{1}{z})$, which is also diagonalized by the the polynomials $\{\ph_n(x)\}_{n=0}^{\infty}$. 
Finally, in \seref{se6} we treat in a detail extensions of the Askey-Wilson polynomials orthogonal with respect to a measure on the real line. We obtain a different derivation and a refinement of the results in \cite{HI}. Applying the main theorem in this situation explains  the existence of some lower-order $q$-difference operators found there in a rather roundabout way. 

\section{Notations} \label{se2}

Throughout the paper we assume that $q$ is a real number in the open interval 
$(0,1)$ and we consider the corresponding $q$-shifted factorials
\begin{equation*}
(a;q)_n=\prod_{l=0}^{n-1}(1-aq^{l}), \quad
(a;q)_{\infty}=\prod_{l=0}^{\infty}(1-aq^{l}),\quad 
(a_1,a_2,\dots,a_j;q)_n=\prod_{l=1}^j(a_l;q)_n,
\end{equation*}
and  ${}_4\phi_3$ basic hypergeometric series
\begin{equation}\label{2.1}
{}_4\phi_3\left[\begin{matrix} a_1 , a_2, a_3, a_4 \\
b_1, b_2, b_3 \end{matrix}\,; q,z\right]
=\sum_{l=0}^{\infty}\frac{(a_1,a_2,a_3,a_4;q)_l}{(b_1,b_2,b_3,q;q)_l}z^l.
\end{equation}
The Askey-Wilson polynomials $p_n(x)=p_n(x;a,b,c,d)$, $n=0,1,2\dots$ introduced in \cite{AW} depend on four free parameters $a,b,c,d$ and can be defined by 
\begin{equation}\label{2.2}
p_n(x;a,b,c,d)=\frac{(ab,ac,ad;q)_n}{a^n}
\fpt{q^{-n}}{abcdq^{n-1}}{az}{az^{-1}}{ab}{ac}{ad},
\end{equation}
where $x=\frac{1}{2}(z+\frac{1}{z})$. Sears' transformation formula \cite[page~49, formula (2.10.4)]{GR} shows that the polynomials $p_n(x;a,b,c,d)$ are symmetric in the four parameters $a,b,c,d$. 
We denote by $E_n$ and $E_n^{-1}$, respectively, the customary forward and backward shift operators, acting on a function $f_n=f(n)$ by 
$$E_nf_n=f_{n+1}, \qquad E_n^{-1}f_n=f_{n-1}.$$
If we define the second-order difference operator 
\begin{equation}\label{2.3}
\cLn=A_n E_n + B_n\Id +C_n E_n^{-1},
\end{equation}
with coefficients
\begin{subequations}\label{2.4}
\begin{align}
A_n&=\frac{1-abcdq^{n-1}}{(1-abcdq^{2n-1})(1-abcdq^{2n})} \label{2.4a}\\
B_n&=\frac{q^{n-1}[(1+abcdq^{2n-1})(sq+s'abcd)-q^{n-1}(1+q)abcd(s+s'q)]}{(1-abcdq^{2n-2})(1-abcdq^{2n})}  \label{2.4b}\\
C_n&=
(1-abq^{n-1})(1-acq^{n-1})(1-adq^{n-1})(1-bcq^{n-1})(1-bdq^{n-1})(1-cdq^{n-1})
                                    \nonumber\\
&\qquad\times \frac{(1-q^{n})}
{(1-abcdq^{2n-2})(1-abcdq^{2n-1})},            \label{2.4c}
\end{align}
\end{subequations}
where $s=a+b+c+d$, $s'=a^{-1}+b^{-1}+c^{-1}+d^{-1}$, then we have
\begin{equation}\label{2.5}
\cLn\, p_n(x)=2xp_n(x).
\end{equation}
Note that the action of $\cLn$ is well defined on functions on $\Nset_0$ (since $C_0=0$). It will be convenient later to use a natural bi-infinite extension of $\cLn$ which acts on functions defined on $\Zset$. We can do this by the formal change of variables $n\rightarrow \ga=n+\ep$ in the coefficients of $\cLn$. We shall consider $\ep$ for which the coefficients $A_{\ga}$, $B_{\ga}$, $C_{\ga}$ are well-defined (i.e. the denominators do not vanish) and $A_{\ga}\neq0$, $C_{\ga}\neq0$ for $n\in\Zset$. We shall denote the bi-infinite extension of $\cLn$ by $\cLga$.  We use the same convention for other difference operators acting on $n$, i.e. if 
$$\cM_n=\sum_{\text{finitely many } j}M^{(j)}_n\,E_n^{j},$$
is a difference operator acting on functions defined on $\Nset_0$, we denote by 
$$\cM_{\ga}=\sum_{\text{finitely many } j}M^{(j)}_{\ga}\,E_{\ga}^{j},$$
its bi-infinite extension acting on functions defined on $\Zset$, where $\ga=n+\ep$ and $E_{\ga}=E_{n}$.

Let us denote by $D_z$ and $D_z^{-1}$, respectively, the forward and backward
$q$-shift operators, acting on a function $h(z)$ by
\begin{equation*}
D_z h(z)=h(qz)\quad\text{and}\quad D_z^{-1}h(z)=h(z/q),
\end{equation*}
and let $\cBz$ be the Askey-Wilson second-order $q$-difference
operator
\begin{equation}                        \label{2.6}
\cBz = A(z)D_z-\left[A(z)+A(1/z)\right]\Id +A(1/z)D_z^{-1},
\end{equation}
where
\begin{equation}                        \label{2.7}
A(z)=\frac{(1-az)(1-bz)(1-cz)(1-dz)}{(1-z^2)(1-qz^2)}.
\end{equation}
Then 
\begin{equation}\label{2.8}
\cBz\, p_n(x)=\la_n p_n(x),
\end{equation}
where
\begin{equation}\label{2.9}
\la_n=(q^{-n}-1)(1-abcdq^{n-1}).
\end{equation}
We also set
\begin{equation}\label{2.10}
\xia_n=\frac{q^n(abcd/q;q)_{n}}{a^n(bc,bd,cd,q;q)_n}=\frac{q^n(abcd/q,bcq^n,bdq^n,cdq^n,q^{n+1};q)_{\infty}}{a^n(abcdq^{n-1},bc,bd,cd,q;q)_{\infty}}.
\end{equation}
The first formula will be sufficient for most applications; the second representation will be used when we want to consider a meromorphic extension of $\xia_n$ in $n$ (and, in particular, when we want to make the change of variable $n\rightarrow \ga$).
Note that $\xia_n$ is symmetric in $b,c,d$ and we denote by $\xib_n$, $\xic_n$, $\xid_n$ the analogous expressions, where the parameters $(a,b,c,d)$ are replaced by $(b,c,d,a)$, $(c,d,a,b)$ and $(d,a,b,c)$, respectively. 

Using the ${}_4\phi_3$ representation of $p_n(x)$ given in \eqref{2.2}, one can show that the polynomials 
$\xia_np_n(x)$ satisfy the following $q$-difference--difference equation
\begin{equation}\label{2.11}
\cBz^{a}(\xia_{n}p_{n}(x)-\xia_{n-1}p_{n-1}(x))=(q^{-n + 1} - abcdq^{n - 1})(\xia_{n}p_{n}(x)+\xia_{n-1}p_{n-1}(x)),
\end{equation}
where 
\begin{equation}\label{2.12}
\cBz^{a}=A^{a}(z)D_z+\left[q-\frac{abcd}{q}-A^{a}(z)-A^{a}(1/z)\right]\Id +A^{a}(1/z)D_z^{-1},
\end{equation}
and
\begin{equation*}
A^{a}(z)=\frac{(q+az)(1-bz)(1-cz)(1-dz)}{(1-z^2)(1-qz^2)}.
\end{equation*}
We denote by $\cBz^{b}$, $\cBz^{c}$, $\cBz^{d}$ the operators analogous to $\cBz^{a}$, where the parameters $(a,b,c,d)$ are replaced by $(b,c,d,a)$, $(c,d,a,b)$ and $(d,a,b,c)$, respectively. Thus equation \eqref{2.11} holds for any permutation of the parameters $(a,b,c,d)$.

We use $\Wr_n$ to denote the discrete Wronskian (or Casorati determinant) of functions $h^{(1)}_n,h^{(2)}_n,\dots,h^{(k)}_n$ of a discrete variable $n$, i.e.
$$\Wr_n(h^{(1)}_n,h^{(2)}_n,\dots,h^{(k)}_n)=\det(h^{(i)}_{n-j+1})_{1\leq i,j\leq k}.$$

\section{Statement of the main result} \label{se3}

Consider now arbitrary polynomials $\phi^{(j)}(\la_n)\in\Cset[\la_n]$, for $j=1,2,\dots,k$, and for each 
$j$ pick $\de_j\in\{a,b,c,d\}$. We set 
\begin{equation}\label{3.1}
\psi^{(j)}_n=\frac{\phi^{(j)}(\la_n)}{\xi^{\de_j}_n}.
\end{equation}
We shall assume that $\psi^{(j)}_n$ are independent (as functions of $n$). 
We define a new family of polynomials in $x$ by the Wronskian formula
\begin{subequations}
\begin{equation}\label{3.2a}
\ph_n(x)=\chi_n\,\Wr_n(\psi^{(1)}_n,\psi^{(2)}_n,\dots,\psi^{(k)}_n,p_n(x)),
\end{equation}
where $\chi_n$ is an appropriate normalizing factor (depending only on $n$). Note the resemblance between the above formula and the superposition law for wobbly solitons of discrete versions of the Kadomtsev-Petviashvili hierarchy. Indeed, if we replace $\xi^{\de_j}_n$ and $p_n(x)$ by appropriate discrete or $q$-exponents, we obtain the Wronskian formula for the wave function of the discrete KP equations \cite[formula (5.3.10), page 84]{I1} or their $q$-deformations \cite[formula (3.6), page 166]{I4}.

The main result is that the polynomials $\ph_n(x)$ are simultaneously diagonalized by two commutative algebras. The first algebra consists of difference operators acting on the variable $n$ (with coefficients independent of $z$), while the second algebra consists of $q$-difference operators acting on the variable $z$ (with coefficients independent of $n$), where $x=\frac{1}{2}(z+\frac{1}{z})$, thus extending the bispectral equations \eqref{2.5} and \eqref{2.8} satisfied by the Askey-Wilson polynomials. The factor $\chi_n$ is not essential for the bispectrality, but we shall fix it below in order to obtain well-defined polynomials for all $n\in\Nset_0$ and to formulate the precise statement.

First, note that if we multiply the $j$-th row in the determinant on the right-hand side of equation \eqref{3.2a} by $\xi^{a}_{n-j+1}$, for $j=1,\dots,k+1$ we can rewrite the formula for $\ph_n(x)$ as follows
\begin{equation}\label{3.2b}
\ph_n(x)=\chi_n\,\prod_{j=1}^{k+1}\frac{1}{\xi^{a}_{n-j+1}}\,\Wr_n(\xi_n^{a}\psi^{(1)}_n,\xi_n^{a}\psi^{(2)}_n,\dots,\xi_n^{a}\psi^{(k)}_n,\xi_n^{a}p_n(x)).
\end{equation}
\end{subequations}
Let $\tau_n$ denote the  principal $k\times k$ sub-determinant 
\begin{equation}\label{3.3}
\tau_n=\chi_n \prod_{j=1}^{k+1}\frac{1}{\xi^{a}_{n-j+1}} \, \Wr_n(\xi_n^{a}\psi^{(1)}_n,\xi_n^{a}\psi^{(2)}_n,\dots,\xi_n^{a}\psi^{(k)}_n).
\end{equation}
If $\de_j=a$, then from \eqref{3.1} we see that in the $j$-th column in the above formula we have Laurent polynomials of $q^n$. If $j=b$ we can pull in front of the determinant the factor $\xi^{a}_{n-j+1}/\xi^{b}_{n-j+1}$ and the remaining entries will be rational functions of $q^n$. It is easy to see that if we pick 
\begin{equation}\label{3.4}
\eta^{b,j}_n=\frac{1}{q^{n(k-1)}}(bcq^{n-j+1},bdq^{n-j+1};q)_{j-1}\,(acq^{n-k+1},adq^{n-k+1};q)_{k-j},
\end{equation}
then a  multiplication by $\eta^{b,j}_n$ cancels all denominators in the $j$-th column and the entries become Laurent polynomials of $q^n$. We define $\eta^{c,j}_n$ and $\eta^{d,j}_n$ by replacing in \eqref{3.4} the parameters $(a,b,c,d)$ with $(a,c,b,d)$ and $(a,d,c,b)$, respectively (i.e. we keep $a$ fixed, and we exchange the roles of $b$ and $c$, or $b$ and $d$).  Thus, if we define 
\begin{equation}\label{3.5}
\chi_n=\prod_{j=1}^{k+1}\xi^{\de_{j}}_{n-j+1}\,\prod_{\begin{subarray} \,1 \leq j\leq k \\ \de_j\neq a \end{subarray}} \eta^{\de_j,j}_n, 
\end{equation}
where we set $\de_{k+1}=a$, then $\tau_n$ becomes a Laurent polynomials of $q^n$. 

Throughout the paper we shall assume that 
\begin{subequations}\label{3.6}
\begin{equation}\label{3.6a}
\tau_{n}\neq0 \text{ for }n=-1,0,1,\dots,
\end{equation}
which is true for generic polynomials $\phi^{(j)}(\la_n)$, and 
\begin{equation}\label{3.6b}
\frac{\chi_{n}\xia_{n-k-1}}{\chi_{n-1}}\neq 0 \text{ for }n=0,1,\dots,
\end{equation}
which is true for generic parameters $(a,b,c,d)$.

\end{subequations}
Next, we define two algebras $\fAz$ and $\fAn$ of Laurent polynomials in $z$ and $q^n$, respectively.
Let 
\begin{equation}\label{3.7}
W=\Span\{\psi^{(1)}_{\ga},\psi^{(2)}_{\ga},\dots,\psi^{(k)}_{\ga}\}
\end{equation} 
be the space spanned by the functions $\psi^{(j)}_{\ga}$ and we denote by $\fAz$ the subalgebra of $\Cset[z+1/z]$ consisting of all polynomials $f(z+1/z)$, for which the space $W$ is $f(\cLga)$-invariant, i.e. 
\begin{equation}\label{3.8}
\fAz=\{f(z+1/z)\in\Cset[z+1/z]:f(\cLga)W\subset W\}.
\end{equation}
The fact that $\fAz$ contains a polynomial of every sufficiently large degree follows from \prref{pr4.1}. Interesting examples of the above construction when $\fAz=\Cset[z+1/z]$ and the polynomials $\{\ph_n(x)\}_{n=0}^{\infty}$ are orthogonal with respect to a measure on $\Rset$ are discussed in \seref{se6}.

Next we define the algebra $\fAn$ which consists of all polynomials $h(\la_{n-k/2})\in\Cset[\la_{n-k/2}]$ such that 
$h(\la_{n-k/2})-h(\la_{n-k/2-1})$ is divisible by $\tau_{n-1}$ in $\Cset[q^{n},q^{-n}]$, i.e.
\begin{equation}\label{3.9}
\fAn=\left\{h(\la_{n-k/2})\in\Cset[\la_{n-k/2}]:\frac{h(\la_{n-k/2})-h(\la_{n-k/2-1})}{\tau_{n-1}} \in\Cset[q^{n},q^{-n}]\right\}.
\end{equation}
The fact that $\fAn$ contains a polynomial of every sufficiently large degree follows from \prref{pr5.1}.

The main result of the paper is the following theorem.
\begin{Theorem}\label{th3.1}
For every $f(z+1/z)\in\fAz$ and every $h(\la_{n-k/2})\in\fAn$ there exist a difference operator $\cLh^f_n$ acting on $n$, which is independent of $z$, and a $q$-difference operator $\cBh^h_z$ acting on $z$, which is independent of $n$, such that
\begin{subequations}\label{3.10}
\begin{align}
\cLh^f_n\ph_n(x)&=f(z+1/z)\ph_n(x),\label{3.10a}\\
\cBh^h_z\ph_n(x)&=h(\la_{n-k/2})\ph_n(x).\label{3.10b}
\end{align}
\end{subequations}
Thus, 
$\fDn=\{\cLh^f_n:f\in \fAz\}$ is a commutative algebra of difference operators in $n$ isomorphic to $\fAz$, $\fDz=\{\cBh^h_z:h\in \fAn\}$ is a commutative algebra of $q$-difference operators in $z$ isomorphic to $\fAn$, and both algebras are diagonalized by the polynomials $\ph_n(x)$.
\end{Theorem}
We establish equations \eqref{3.10a} and \eqref{3.10b} in Sections \ref{se4} and \ref{se5}, respectively.

\section{The difference equations in $n$} \label{se4}

Consider the bi-infinite extension $\cLga$ of the Askey-Wilson recurrence operator defined by equations \eqref{2.3}-\eqref{2.4} and for each $\ell\in\{a,b,c,d\}$ we denote by $\cLga^{\ell}$ the difference operator obtained by conjugating $\cLga$ with the operator multiplication by $\xi^{\ell}_{\ga}$, i.e.
\begin{equation}\label{4.1}
\cLga^{\ell}=\xi^{\ell}_{\ga}\,\cLga \frac{1}{\xi^{\ell}_{\ga}}.
\end{equation}
Explicitly, when $\ell=a$ we can write $\cLga^{a}$ as follows
\begin{equation}\label{4.2}
\cLga^a=A^{a}_{\ga}\, E_{\ga} + B^{a}_{\ga}\,\Id +C^{a}_{\ga}\, E_{\ga}^{-1},
\end{equation}
where
\begin{subequations}\label{4.3}
\begin{align}
A^{a}_{\ga}&=\frac{a(1-bcq^{\ga})(1-bdq^{\ga})(1-cdq^{\ga})(1-q^{\ga+1})}{q(1-abcdq^{2\ga-1})(1-abcdq^{2\ga})}, \label{4.3a}\\
B^{a}_{\ga}&=\frac{a}{q}+\frac{q}{a}-A^{a}_{\ga}-C^{a}_{\ga}, \label{4.3b}\\
C^{a}_{\ga}&= \frac{q(1-abcdq^{\ga-2})(1-abq^{\ga-1})(1-acq^{\ga-1})(1-adq^{\ga-1})}{a(1-abcdq^{2\ga-2})(1-abcdq^{2\ga-1})}.            \label{4.3c}
\end{align}
\end{subequations}
From these formulas it is not hard to see that for every $j\in\Nset_0$ we have
$$\cLga^a\, (\la_{\ga})^j=\left(\frac{a}{q^{j+1}}+\frac{q^{j+1}}{a}\right)(\la_{\ga})^j+\text{ a polynomial in $\la_{\ga}$ of degree at most $(j-1)$,}$$
where $\la_{\ga}$ is the eigenvalue of the operator $\cBz$ defined in \eqref{2.9}. The above equation shows that every element in $\Cset[\la_{\ga}]$ is annihilated by an appropriated polynomial of $\cLga^{a}$. Combining this with \eqref{3.1}, \eqref{3.7} and \eqref{4.1} we obtain the following statement.

\begin{Proposition}\label{pr4.1}
There exists a nonzero polynomial $f$ such that $f(\cLga)W=\{0\}$. In particular, the algebra $\fAz$ defined by \eqref{3.8} contains a polynomial of every sufficiently large degree.  
\end{Proposition}

We can easily write an explicit eigenvector for the operator $\cLga^a$ in the space $\Cset[\la_{\ga}]$ for the eigenvalue $a/q^{j+1}+q^{j+1}/a$.

\begin{Proposition}\label{pr4.2}
For $j\in\Nset_0$ define 
\begin{equation}\label{4.4}
g^{a}_j(\la_{\ga})=(q^2/ab,q^2/ac,q^2/ad;q)_{j}\,\fpt{q^{-j}}{q^{j+2}/a^2}{q^{\ga+1}}{1/abcdq^{\ga-2}}{q^2/ab}{q^2/ac}{q^2/ad}.
\end{equation}
Then 
\begin{equation}\label{4.5}
\cLga^a\, g^{a}_j(\la_{\ga})=\left(\frac{a}{q^{j+1}}+\frac{q^{j+1}}{a}\right)g^{a}_j(\la_{\ga}).
\end{equation}
\end{Proposition}
\begin{proof}
Using the definition \eqref{2.1} of the ${}_4\phi_3$ basic hypergeometric series, one can check directly that \eqref{4.5} holds. 
\end{proof}
Note that if $a\notin\{\pm q^{1+s/2}:s\in\Nset_0\}$, then $g^{a}_j(\la_{\ga})$ is a polynomial of degree $j$ in $\la_{\ga}$. 
Moreover, in this case the eigenvalues $a/q^{j+1}+q^{j+1}/a$ are different for $j\in\Nset_0$, the operator $\cLga^a$ can be diagonalized in $\Cset[\la_{\ga}]$, and the polynomials given in \eqref{4.4} provide an explicit eigenbasis. Clearly, we obtain similar statements for the operators $\cLga^{b}$, $\cLga^{c}$, $\cLga^{d}$ and we denote by $g^{b}_j(\la_{\ga})$, $g^{c}_j(\la_{\ga})$, $g^{d}_j(\la_{\ga})$, respectively the corresponding polynomials in formula \eqref{4.4}.

\begin{Remark}\label{re4.3}
We note that the polynomial $g^{a}_j(\la_{\ga})$ can also be obtained from \cite{IR}. Indeed, appropriately normalized,  the polynomial $g^{a}_j(\la_{\ga})$ in equation \eqref{4.4} above corresponds to the eigenfunction $s_{\ga}(aq^{-j-1})$ of the bi-infinite Askey-Wilson recurrence operator given in formula (1.13) on page 203 in \cite{IR}. To see this, we need to multiply $s_{\ga}(aq^{-j-1})$ by an appropriate factor depending on $\ga$ (because the recurrence operator there differs from $\cLga^{a}$). Then we  apply first the Watson's transformation formula \cite[formula III.18]{GR} to rewrite the ${}_8\phi_7$ series as ${}_4\phi_3$ series, and then we need to use Sears' formula \cite[formula III.16]{GR} to transform the ${}_4\phi_3$ series.
\end{Remark}

Consider now the difference operator $\cQ_{\ga}$ defined by 
\begin{equation}\label{4.6}
\cQ_{\ga}f_{\ga}=\chi_{\ga}\,\Wr_{\ga}(\psi^{(1)}_{\ga},\psi^{(2)}_{\ga},\dots,\psi^{(k)}_{\ga},f_{\ga}).
\end{equation}
Note that $W=\ker(\cQ_{\ga})$ and therefore, for $f\in\fAz$ we have 
$$f(\cLga)\ker(\cQ_{\ga})\subset \ker(\cQ_{\ga}),$$ 
which implies that 
$$\ker(\cQ_{\ga})\subset \ker(\cQ_{\ga}f(\cLga)).$$
This means that there exists a difference operator $\cLh_{\ga}^{f}$ with the same support as $f(\cLga)$ such that 
\begin{equation}\label{4.7}
\cLh_{\ga}^{f}\cQ_{\ga}=\cQ_{\ga}f(\cLga).
\end{equation}
We can write $\cQ_{\ga}$ as follows 
$$\cQ_{\ga}=\sum_{l=0}^{k}Q^{(-l)}_{\ga}E_{\ga}^{-l},$$
where $Q^{(-l)}_{\ga}$, $l=0,\dots,k$ denote the coefficients in front of the shift operators $E_{\ga}^{-l}$. From the explicit formulas \eqref{2.10}, \eqref{3.4} and \eqref{3.5} we see that for $l=1,2,\dots,k$
\begin{equation}\label{4.8}
Q^{(-l)}_{\ga}=0\text{ for }\ga=0,1,\dots,l-1.
\end{equation}
Moreover, $Q^{(0)}_{\ga}=(\chi_{\ga}\xia_{\ga-k-1}/\chi_{\ga-1})\tau_{\ga-1}$ and therefore 
equations \eqref{3.6} show that $Q^{(0)}_{\ga}\neq 0$ for $\ga\in\Nset_0$. This implies that the coefficients of the semi-infinite operator $\cLh_{n}^{f}$ are well-defined for $n\in\Nset_0$ and we have 
\begin{equation}\label{4.9}
\cLh_{n}^{f}\cQ_{n}=\cQ_{n}f(\cL_n).
\end{equation}
Since $\ph_n(x)=\cQ_n p_n(x)$, we see that equations \eqref{2.5} and \eqref{4.9} imply equation \eqref{3.10a}.

\section{The $q$-difference equations in $z$} \label{se5}
\subsection{Auxiliary facts}\label{ss5.1}

We define the involution $I_{n}^{(j)}$ on the algebra $\Cset(q^{n})$ of rational functions of $q^n$ by 
$$I_{n}^{(j)}(q^{n})=\frac{1}{abcdq^{n-j-1}}.$$
From \eqref{2.9} it is easy to see that 
\begin{equation*}
I_{n}^{(j)}(\la_{n-l})=\la_{n+l-j}.
\end{equation*}
In particular, $I_{n}^{(j)}(\la_{n-j/2})=\la_{n-j/2}$ hence every polynomial of $\la_{n-j/2}$ is 
invariant under the action of $I_{n}^{(j)}$. Conversely, if $r\in\Cset[q^{n},q^{-n}]$ is a Laurent polynomial of $q^{n}$, invariant under $I_{n}^{(j)}$, then $r\in\Cset[\la_{n-j/2}]$.

Using the above, we can show that, up to a simple factor, $\tau_n$ defined by equations \eqref{3.3}-\eqref{3.5} 
belongs to $\Cset[\la_{n-(k-1)/2}]$. Indeed, as we noted in \seref{se3}, we can re-write $\tau_n$ as a $k\times k$ determinant whose entries are Laurent polynomials of $q^{n}$. Moreover, one can show that $I_{n}^{(k-1)}$ reverses the order of the rows (i.e. it maps the $l$-th row into the $k-l+1$ for $l=1,2,\dots,k$). Indeed, if $\de_j=a$ then the $(l,j)$ entry of the determinant which represents $\tau_n$ is $\phi^{(j)}(\la_{n-l+1})$ and we have 
\begin{equation}\label{5.1}
I_{n}^{(k-1)}(\phi^{(j)}(\la_{n-l+1}))=\phi^{(j)}(\la_{n+l-k}).
\end{equation}
If, say, $\de_j=b$ then we need to multiply the $j$-th column of the Wronskian formula \eqref{3.3} by $\eta_n^{b,j}\xib_{n-j+1}/\xia_{n-j+1}$.  The $(l,j)$ entry of the resulting determinant is 
$\ka_n^{b;l,j}\phi^{(j)}(\la_{n-l+1})$, where 
\begin{subequations}\label{5.2}
\begin{equation}\label{5.2a}
\ka_n^{b;l,j}=\eta_n^{b,j}\,\frac{\xib_{n-j+1}}{\xia_{n-j+1}}\frac{\xia_{n-l+1}}{\xib_{n-l+1}}.
\end{equation}
Using the explicit formulas \eqref{2.10} and \eqref{3.4} it is easy to see that 
\begin{equation}\label{5.2b}
\ka_n^{b;l,j}=\frac{1}{q^{n(k-1)}}\left(\frac{b}{a}\right)^{j-l}(acq^{n-k+1},adq^{n-k+1};q)_{k-l}\,
(bcq^{n-l+1},bdq^{n-l+1};q)_{l-1}.
\end{equation}
\end{subequations}
A straightforward computation now shows that 
\begin{equation}\label{5.3}
I_{n}^{(k-1)}(\ka_n^{b;l,j})=\ka_n^{b;k-l+1,j}
\end{equation} 
and the cases $\de_j=c$ or $\de_j=d$ are similar. 

Thus, $I_{n}^{(k-1)}$ reverses the order of the rows and therefore
\begin{equation*}
I_{n}^{(k-1)}(\tau_{n})=(-1)^{k(k-1)/2}\tau_n.
\end{equation*}
This formula shows that if $k\equiv 0,1\mod 4$, then 
$\tau_n\in\Cset[\la_{n-(k-1)/2}]$. Otherwise, $\tau_n$ is divisible by 
$$\la_{n-k/2+1}-\la_{n-k/2}=(1-q)q^{k/2-1}(q^{-n}-abcdq^{n-k})$$ 
and the quotient belongs 
to $\Cset[\la_{n-(k-1)/2}]$. Therefore, we see that 
\begin{equation}\label{5.4}
\tau_n=\ep^{(k)}_n \taub(\la_{n-(k-1)/2}), 
\end{equation}
where $\taub$ is a polynomial and
\begin{equation}\label{5.5}
\ep^{(k)}_n=\begin{cases}
1& \text{ if }k\equiv 0,1\mod 4\\
\la_{n-k/2+1}-\la_{n-k/2}& \text{ if }k\equiv 2,3\mod 4.
\end{cases}
\end{equation}
Next we want to characterize the elements of the algebra $\fAn$ defined by \eqref{3.9} and to show that  
$\fAn$ contains a polynomial of every degree greater than $\deg(\taub)$. We can define the elements of $\fAn$ as discrete integrals of Laurent polynomials divisible by $\tau_n$. To make this more precise, we introduce the following notation. For $n,m\in\Zset$ and for a function $f_s$ defined on $\Zset$ it will 
be convenient to define the discrete integral
\begin{equation*}
\int_{m}^{n}f_s\dms=\begin{cases}
\sum_{s=m+1}^{n}f_s &\text{if }n>m\\
0 &\text{if }n=m\\
-\sum_{s=n+1}^{m}f_s&\text{if }n<m.
\end{cases}
\end{equation*}
Thus 
\begin{equation}\label{5.6}
\int_m^{n}f_s\dms=f_n+\int_m^{n-1}f_s\dms\text{ for all }m, n\in\Zset.
\end{equation}
Note  that for every polynomial $r$ there exists $\bar{r}\in\Cset[t]$ such that
\begin{equation*}
\frac{r(\la_{n-k/2})-r(\la_{n-k/2-1})}{\la_{n-k/2}-\la_{n-k/2-1}}=\bar{r}(\la_{n-(k+1)/2}).
\end{equation*}
The linear map $r(t)\rightarrow \bar{r}(t)$ is a discrete analog of the differentiation operator. It is easy to see that it maps $\Cset[t]$ onto itself and the kernel consists of the constants. Thus, for every polynomial $\bar{r}$ there exists a polynomial $r$ 
such that the above equation holds. Moreover, up to an additive constant, 
$r$ is uniquely determined by 
\begin{equation*}
\begin{split}
r(\la_{n-k/2})&=\int_{0}^{n}(\la_{s-k/2}-\la_{s-k/2-1})\bar{r}(\la_{s-(k+1)/2})\dms\\
&=\int_{-1}^{n-1}(\la_{s-k/2+1}-\la_{s-k/2})\bar{r}(\la_{s-(k-1)/2})\dms.
\end{split}
\end{equation*}
Summarizing the above, we obtain the following characterization of $\fAn$.
\begin{Proposition}\label{pr5.1}
For a polynomial $h$ we have $h(\la_{n-k/2})\in\fAn$ if and only if 
\begin{equation}\label{5.7}
h(\la_{n-k/2})=\int_{0}^{n}\ep^{(k+2)}_{s}g(\la_{s-(k+1)/2})
\tau_{s-1}\dms+c_0,
\end{equation}
where $g$ is a polynomial and $c_0$ is a constant. In particular, $\fAn$ contains a polynomial of every degree greater than $\deg(\taub)$.
\end{Proposition}
For the proof of \eqref{3.10b} we need two more facts. 
First we formulate a discrete analog of a lemma due to Reach \cite{R}, which was derived and used in \cite{I1} to establish the bispectral properties of specific rational solution of the Today lattice hierarchy. It was used more recently in \cite{I2,I3} within the context of orthogonal polynomials satisfying higher-order differential equations.

\begin{Lemma}\label{le5.2}
Let $f^{(0)}_{n}, f^{(1)}_{n},\dots,f^{(k+1)}_{n}$ be functions of a discrete 
variable $n$. Fix $n_1,n_2,\dots,n_{k+1}\in\Zset$ and let 
\begin{equation}\label{5.8}
F_n=\sum_{j=1}^{k+1}(-1)^{k+1+j}f^{(j)}_n\int_{n_j}^{n}f^{(0)}_{s}
\Wr_s(f^{(1)}_{s},\dots,\fh^{(j)}_{s},
\dots,f^{(k+1)}_{s})\dms,
\end{equation}
with the usual convention that the terms with hats are omitted. 
Then 
\begin{equation}\label{5.9}
\begin{split}
\Wr_n(f^{(1)}_{n},\dots,f^{(k)}_{n},F_n)=&
\int_{n_{k+1}}^{n-1}f^{(0)}_{s}\Wr_s(f^{(1)}_{s},\dots,f^{(k)}_{s})\dms\\
&\times \Wr_n(f^{(1)}_{n},\dots,f^{(k+1)}_{n}).
\end{split}
\end{equation}
\end{Lemma}
Since the application in our case is very subtle and we need different elements of the proof, we 
briefly sketch it below. 
 
\begin{proof}[Proof of \leref{le5.2}]
Note that
\begin{equation*}
\left|\begin{matrix}
f^{(1)}_{n}&f^{(2)}_{n}&\dots &f^{(k+1)}_{n}\\
f^{(1)}_{n-1}&f^{(2)}_{n-1}&\dots &f^{(k+1)}_{n-1}\\
\vdots & \vdots& &\vdots\\
f^{(1)}_{n-k+1}&f^{(2)}_{n-k+1}&\dots &f^{(k+1)}_{n-k+1}\\
f^{(1)}_{n-l}&f^{(2)}_{n-l}&\dots &f^{(k+1)}_{n-l}
\end{matrix}\right|=0, \text{ for every }l=0,1,\dots,k-1.
\end{equation*}
Expanding the above determinant along the last row we obtain
\begin{equation}\label{5.10}
\sum_{j=1}^{k+1}(-1)^{k+1+j}f^{(j)}_{n-l}
\Wr_n(f^{(1)}_{n},\dots\fh^{(j)}_{n},\dots,f^{(k+1)}_{n})=0
\text{ for }l=0,\dots,k-1.
\end{equation}
Using \eqref{5.6} and \eqref{5.10} we see that for $l=0,\dots,k-1$ we have
\begin{equation}\label{5.11}
F_{n-l}=\sum_{j=1}^{k+1}(-1)^{k+1+j}f^{(j)}_{n-l}\int_{n_j}^{n}f^{(0)}_{s}
\Wr_s(f^{(1)}_{s},\dots,\fh^{(j)}_{s},
\dots,f^{(k+1)}_{s})\dms,
\end{equation}
and 
\begin{equation}\label{5.12}
\begin{split}
F_{n-k}=&\sum_{j=1}^{k+1}(-1)^{k+1+j}f^{(j)}_{n-k}\int_{n_j}^{n}f^{(0)}_{s}
\Wr_s(f^{(1)}_{s},\dots,\fh^{(j)}_{s},
\dots,f^{(k+1)}_{s})\dms \\
&\qquad-f^{(0)}_{n}\Wr_n(f^{(1)}_{n},\dots,f^{(k+1)}_{n}).
\end{split}
\end{equation}
If we plug \eqref{5.11} and \eqref{5.12} in 
$\Wr_n(f^{(1)}_{n},\dots,f^{(k)}_{n},F_n)$, then most of the terms cancel by 
column elimination and we obtain \eqref{5.9}.
\end{proof}

\begin{Remark}\label{re5.3}
We list below important corollaries from the proof of \leref{le5.2}.\\

\noindent
\item[(i)] Note that the right-hand side of \eqref{5.9} does not depend 
on the integers $n_1,n_2\dots,n_k$. Moreover, if we change $n_{k+1}$ then 
only the value of 
$$\int_{n_{k+1}}^{n-1}f^{(0)}_{s}\Wr_s(f^{(1)}_{s},\dots,f^{(k)}_{s})\dms$$
will change by an additive constant, which is independent of $n$ and 
$f^{(k+1)}_{n}$. Thus, instead of \eqref{5.8} we can write 
\begin{equation*}
F_n=\sum_{j=1}^{k+1}(-1)^{k+1+j}f^{(j)}_n\int^{n}f^{(0)}_{s}
\Wr_s(f^{(1)}_{s},\dots,\fh^{(j)}_{s},
\dots,f^{(k+1)}_{s})\dms,
\end{equation*}
leaving the lower bounds of the integrals (sums) blank and we can fix them 
at the end appropriately. This would allow us to easily change the variable, 
without keeping track of the lower end.\\

\noindent
\item[(ii)] From \eqref{5.10} it follows that for every $l=-1, 0,1\dots,k-1$ we 
can write $F_n$ also as
\begin{equation*}
F_n=\sum_{j=1}^{k+1}(-1)^{k+1+j}f^{(j)}_n\int^{n+l}f^{(0)}_{s}
\Wr_s(f^{(1)}_{s},\dots,\fh^{(j)}_{s},
\dots,f^{(k+1)}_{s})\dms,
\end{equation*}
and changing the variable in the discrete integral we obtain
\begin{equation*}
F_n=\sum_{j=1}^{k+1}(-1)^{k+1+j}f^{(j)}_n\int^{n}f^{(0)}_{s+l}
\Wr_s(f^{(1)}_{s+l},\dots,\fh^{(j)}_{s+l},
\dots,f^{(k+1)}_{s+l})\dms.
\end{equation*}
\item[(iii)] Suppose now that $k$ is even.
Using (ii) above with $l=k/2-1$ we have
\begin{equation}\label{5.13}
F_n=\sum_{j=1}^{k+1}(-1)^{k+1+j}f^{(j)}_n\int^{n}f^{(0)}_{s+k/2-1}
\Wr_s(f^{(1)}_{s+k/2-1},\dots,\fh^{(j)}_{s+k/2-1},
\dots,f^{(k+1)}_{s+k/2-1})\dms.
\end{equation}
Let us consider the sum consisting of the first $k$ integrals:
\begin{equation*}
F_n^{(k)}=\sum_{j=1}^{k}(-1)^{k+1+j}f^{(j)}_n\int^{n}f^{(0)}_{s+k/2-1}
\Wr_s(f^{(1)}_{s+k/2-1},\dots,\fh^{(j)}_{s+k/2-1},
\dots,f^{(k+1)}_{s+k/2-1})\dms.
\end{equation*}
Expanding each Wronskian determinant along the last column we can write 
$F_n^{(k)}$ as a sum of $k$ terms $F_n^{(k,m)}$, each one involving as 
integrand one of the functions $f^{(k+1)}_{s+m}$, where 
$m=-k/2,-k/2+1,\dots,k/2-1$.
We can use \eqref{5.10} once again, this time for the functions 
$f^{(0)}_{n}, f^{(1)}_{n},\dots,f^{(k)}_{n}$ (i.e. omitting $f^{(k+1)}_{n}$),
to change $s$ as follows:\\
\begin{itemize}
\item If $m\geq 0$ we can replace $n$ with $n-m$ in the upper limit of the 
integral, or equivalently, if we keep the upper limit of the integral to be 
$n$, we can replace $s$ by $s-m$ in the integrand. Thus $F_n^{(k,m)}$ will 
have $f^{(k+1)}_{s}$ 
as integrand (and the integration goes up to $n$).
\item If $m\leq -1$ we can replace $s$ by $s-m-1$, thus $F_n^{(k,m)}$ will have 
$f^{(k+1)}_{s-1}$ as integrand (and the integration goes up to $n$).
\end{itemize}
This means that we can rewrite $F^{(k)}_{n}$ as sums of integrals, and the 
integrands can be combined in pairs (corresponding to $m$ and $(-m-1)$ 
for $m=0,1,\dots,k/2-1$) involving $f^{(k+1)}_{s}$ and 
$f^{(k+1)}_{s-1}$. Explicitly, we can write $F_n^{(k)}$, as a sum of terms 
which, up to a sign, have the form
\begin{subequations}\label{5.14}
\begin{equation}\label{5.14a}
f^{(j)}_n\int^{n}\Bigg[f^{(0)}_{s-m+k/2-1}\De^{1,j}_{s}f^{(k+1)}_{s}
-f^{(0)}_{s+m+k/2-1}\De^{2,j}_sf^{(k+1)}_{s-1}\Bigg]\dms,
\end{equation}
for $m=0,1,\dots,k/2-1$ and $j=1,2,\dots,k$, where 
\begin{equation}\label{5.14b}
\De^{1,j}_{n}=\left|\begin{matrix}
f^{(1)}_{n+k/2-m-1}&\dots&\fh^{(j)}_{n+k/2-m-1}&\dots& f^{(k)}_{n+k/2-m-1}\\
f^{(1)}_{n+k/2-m-2}&\dots&\fh^{(j)}_{n+k/2-m-2}&\dots &f^{(k)}_{n+k/2-m-2}\\
\vdots& &\vdots & &\vdots\\
\fh^{(1)}_{n}&\dots&\fh^{(j)}_{n}&\dots &\fh^{(k)}_{n}\\
\vdots& &\vdots\\
f^{(1)}_{n-k/2-m}&\dots &\fh^{(j)}_{n-k/2-m}&\dots &f^{(k)}_{n-k/2-m}
\end{matrix}\right|
\end{equation}
and
\begin{equation}\label{5.14c}
\De^{2,j}_{n}=\left|\begin{matrix}
f^{(1)}_{n+k/2+m-1}&\dots &\fh^{(j)}_{n+k/2+m-1}&\dots &f^{(k)}_{n+k/2+m-1}\\
f^{(1)}_{n+k/2+m-2}&\dots &\fh^{(j)}_{n+k/2+m-2}&\dots &f^{(k)}_{n+k/2+m-2}\\
\vdots& &\vdots & &\vdots\\
\fh^{(1)}_{n-1}&\dots & \fh^{(j)}_{n-1}&\dots&\fh^{(k)}_{n-1}\\
\vdots& &\vdots& &\vdots\\
f^{(1)}_{n-k/2+m}&\dots &\fh^{(j)}_{n-k/2+m}&\dots &f^{(k)}_{n-k/2+m}
\end{matrix}\right|.
\end{equation}
\end{subequations}

\noindent
\item[(iv)] Finally if $k$ is odd, we shall use (ii) with $l=(k-1)/2$ and 
$l=(k-3)/2$ and write $F_n$ as the average of these two sums. Then we can 
apply the same procedure that we used in (iii) to write $F_n^{(k)}$ as sums of 
integrals, whose integrands can be combined in pairs involving 
$f^{(k+1)}_{s}$ and $f^{(k+1)}_{s-1}$.
\end{Remark}
The second important ingredient needed for the proof of \eqref{3.10b} is the 
following results which represents an extension of Lemma 4.5 in \cite{I2}.

\begin{Lemma}\label{le5.4}
For $r(q^n)\in\Cset[q^{n},q^{-n}]$ and $\ell\in\{a,b,c,d\}$ there exists a $q$-difference operator $\cBb_z^{\ell,r}$ such that for $n\in\Nset_0$ we have
\begin{equation}\label{5.15}
\int_{-1}^{n}[r(q^s)\xi^{\ell}_sp_{s}(x)
-r(q^{2-s}/abcd)\xi^{\ell}_{s-1}p_{s-1}(x)]\dms=\cBb_z^{\ell,r}\xi^{\ell}_{n}p_n(x).
\end{equation}
\end{Lemma}

\begin{proof}
Without restriction we can assume that $\ell=a$. Note first that \eqref{5.15} for $\ell=a$ is equivalent to show that for every $r(q^n)\in\Cset[q^{n},q^{-n}]$ there exists a $q$-difference operator $\cBb_z^{a,r}$ such that
\begin{equation}\label{5.16}
r(q^n)\xia_np_{n}(x)
-r(q^{2-n}/abcd)\xia_{n-1}p_{n-1}(x)=\cBb_z^{a,r}[\xi^{a}_{n}p_n(x)-\xi^{a}_{n-1}p_{n-1}(x)].
\end{equation}
From the last formula it is clear that the set $\cA$ of all Laurent polynomials $r$ for which there exists a $q$-difference operator $\cBb_z^{a,r}$ satisfying \eqref{5.16} is a vector space over $\Cset$ that contains the constants. Note also that 
\begin{subequations}\label{5.17}
\begin{align}
\la_n&=(q^{-n}-1)(1-abcdq^{n-1})\in\cA,\label{5.17a}\\
r_0(q^{n})&=q^n-\frac{q^{2-n}}{abcd}=-\frac{q}{abcd}(q^{-n+1}-abcdq^{n-1})\in\cA,\label{5.17b}
\end{align}
\end{subequations}
by \eqref{2.8} and \eqref{2.11}, respectively. Moreover, if $r(q^n)\in\cA$ and $R(q^n)\in\Cset[q^n+q^{2-n}/abcd]\cap\cA$ (i.e. $R(q^n)$ is invariant under the change $q^n\rightarrow q^{2-n}/abcd$), then it is easy to see that $rR\in\cA$, because \eqref{5.16} holds with 
$\cBb_z^{a,rR}=\cBb_z^{a,r}\cBb_z^{a,R}$. Since arbitrary $r(q^n)\in\Cset[q^{n},q^{-n}]$ can be written as
\begin{equation}\label{5.18}
r(q^n)=R^1(q^n)+R^2(q^n)r_0(q^n),
\end{equation}
where $r_0$ is given in \eqref{5.17b} and
\begin{subequations}\label{5.19}
\begin{align}
R^1(q^n)&=\frac{1}{2}\left[r(q^n)+r(q^{2-n}/abcd)\right]\in\Cset[q^n+q^{2-n}/abcd],\label{5.19a}\\
R^2(q^n)&=\frac{r(q^n)-r(q^{2-n}/abcd)}{2(q^n-q^{2-n}/abcd)}\in\Cset[q^n+q^{2-n}/abcd],\label{5.19b}
\end{align}
\end{subequations}
we see that it is enough to prove that $R(q^n)=q^n+q^{2-n}/abcd\in\cA$. But 
this follows from equations \eqref{5.17} and the fact that $\cA$ contains the constant polynomials, since
$$R(q^n)=q^n+\frac{q^{2-n}}{abcd}=\frac{2q^2}{abcd(1+q)}\la_n+\frac{1-q}{1+q}r_0(q^n)+
\frac{2q(abcd+q)}{abcd(1+q)}.$$
\end{proof}

\begin{Remark}\label{re5.5}
Note that the above proof can be used to write an explicit formula for the operator $\cBb_z^{\ell,r}$ in \eqref{5.15} as a polynomial of the operators $\cBz$ and $\cBz^{\ell}$ by representing $r(q^n)$ as a sum of symmetric and skew-symmetric parts given in equation \eqref{5.18}. 
\end{Remark}

\subsection{Construction of the $q$-difference operators in $z$}\label{ss5.2}
 
We are now ready to prove that for every $h\in\fAn$ there exists a $q$-difference operator $\cBh_z^{h}$ such that \eqref{3.10b} holds. From \prref{pr5.1} we know that, up to an additive constant, we have 
\begin{equation*}
h(\la_{n-k/2})=\int_{-1}^{n-1}\ep^{(k+2)}_{s+1}g(\la_{s-(k-1)/2})\tau_{s}\dms.
\end{equation*}
We apply \leref{le5.2} with 
\begin{align*}
f^{(0)}_n&=\frac{\chi_n}{\prod_{j=1}^{k+1}\xia_{n-j+1}}\ep^{(k+2)}_{n+1}g(\la_{n-(k-1)/2}),\\
f^{(j)}_n&=\xia_{n}\psi^{(j)}_n\text{ for }j=1,2,\dots,k,\\
f^{(k+1)}_n&=\xia_{n}p_{n}(x).
\end{align*}
If we multiply equation \eqref{5.9} by ${\chi_n}/{\prod_{j=1}^{k+1}\xia_{n-j+1}}$ then, using \eqref{3.2b} and \eqref{3.3} we see that the right-hand is equal to $(h(\la_{n-k/2})+c_0)\ph_n(x)$, where $c_0$ is a constant 
(independent of $n$ and $x$). The goal now is to show that, if we 
choose appropriately the integers $n_j$ in \leref{le5.2}, then there exists 
a $q$-difference operator $\cBh_z$ (with coefficients independent of $n$) such that 
\begin{equation}\label{5.20}
F_n=\cBh_z \,\xia_n p_n(x).
\end{equation}
Indeed, if this equation is true, then we can pull the operator $\cBh_z$ in front of the Wronskian determinant on the left-hand side of \eqref{5.9} and using \eqref{3.2b} again we will obtain the left-hand side of \eqref{3.10b}, thus completing the proof of \eqref{3.10b} (and, up to an additive constant, the operator $\cBh_z$ is the operator $\cBh_z^{h}$).

Suppose first that $k$ is even and let us write $F_n$ as explained in 
\reref{re5.3} (iii).  Note that the last term in the sum \eqref{5.13} representing $F_n$ is of the form $(h(\la_n)+c_0')\xia_n p_n(x)$, where $c_0'$ is a constant (independent of $n$ and $x$), and using \eqref{2.8} we can rewrite it as $(h(\cBz)+c_0')\xia_n p_n(x)$. 
Thus we can consider the sum $F^{(k)}_{n}$ of the first $k$ terms. 
It is enough to show that each term of the form \eqref{5.14a} can be 
represented as $\cBh_z  \,\xia_np_n(x)$ for some $q$-difference operator $\cBh_z$. We prove this by using \leref{le5.4} with $\ell=\de_j$. Note that 
$$r(q^{2-n}/abcd)=I_{n}^{(1)}(r(q^n)),$$
so the key point is to show that the coefficients in front of $\xi^{\de_j}_{s}p_{s}(x)$ and $\xi^{\de_j}_{s-1}p_{s-1}(x)$ in \eqref{5.14a} are connected by the involution $I_{s}^{(1)}$. It is useful to connect the involution $I^{(1)}_{n}$ to the involution $I^{(k-1)}_{n}$, which we used earlier, by the intertwining relation
\begin{equation}\label{5.21}
E_n^{-(k/2+m-1)}\,I_{n}^{(1)}\,E_n^{k/2-m-1}=I_{n}^{(k-1)}.
\end{equation}
We have two possible cases: $\de_j=a$ or $\de_j\in\{b,c,d\}$. 

{\em Case 1:}  $\de_j=a$ (hence $f^{(j)}_n=\phi^{(j)}(\la_n)$). Note that if we set 
$\fb_n={\chi_n}/{\prod_{j=1}^{k+1}\xia_{n-j+1}}$, then 
\begin{equation}\label{5.22}
I_{n}^{(1)}(\fb_{n-m+k/2-1}\De^{1,j}_n)=(-1)^{(k-1)(k-2)/2}\,\fb_{n+m+k/2-1}\De^{2,j}_n,
\end{equation}
where $\De^{1,j}_n$ and $\De^{2,j}_n$ are the determinants given in equations \eqref{5.14b} and \eqref{5.14c}, respectively. Indeed, the same arguments that we used at the beginning of \seref{se5} show that, after we multiply the $i$-th column of $\De^{1,j}_n$ and $\De^{2,j}_n$  by the corresponding factor ($\eta^{\de_i,i}_n\xi^{\de_i}_{n-i+1}/\xia_{n-i+1}$) from $\fb$ for every $i$ such that $\de_i\neq a$, then $I_{n}^{(1)}$ will reverse the order of the rows in these two determinants (and formula \eqref{5.21} shows that the arguments work in the same way with $I_{n}^{(k-1)}$ there replaced by $I_{n}^{(1)}$ here.). Since $\De^{i,j}_n$ are $(k-1)\times(k-1)$ determinants, we obtain \eqref{5.22}. Note also that 
$$I_{n}^{(1)}(g(\la_{n-m-1/2}))=g(\la_{n+m-1/2})$$
and 
$$I_{n}^{(1)}(\ep^{(k+2)}_{n+k/2-m})=(-1)^{(k-1)(k-2)/2}\,\ep^{(k+2)}_{n+k/2+m}.$$
The last equation, follows easily from \eqref{5.5} by considering separately the two possible cases $k\equiv 0$ mod $4$ and $k\equiv 2$ mod $4$.
Since 
$$f^{(0)}_n=\fb_n\,\ep^{(k+2)}_{n+1}g(\la_{n-(k-1)/2}),$$
we see that \leref{le5.4} with $\ell=a$ can be applied for the integral in \eqref{5.14a} and we can rewrite \eqref{5.14a} as 
$$\phi^{(j)}(\la_n)\cBh_z'\xia_n p_n(x),$$
for some $q$-difference operator $\cBh_z'$. We can now commute $\cBh_z'$ and $\phi^{(j)}(\la_n)$ and use \eqref{2.8} to rewrite the above equation as 
$$\cBh_z'\phi^{(j)}(\cBz)\xia_n p_n(x),$$
completing the proof in this case. \\

{\em Case 2:}  $\de_j\in\{b,c,d\}$. Let us fix $\de_j=b$ (the cases $\de_j=c$ or $\de_j=d$ are similar). Then $f^{(j)}_n=\xia_n\phi^{(j)}(\la_n)/\xib_n$. Note that the $j$-th columns of the determinants $\De^{1,j}_n$ and $\De^{2,j}_n$ are removed and therefore the main difference between this and the previous case is that in $f^{(0)}_n$ we have the extra factor $\eta^{b,j}_n\xib_{n-j+1}/\xia_{n-j+1}$. Using the arguments in case 1, we see that the integral in equation \eqref{5.14a} can be rewritten as 
\begin{align*}
&\int^{n}\Big[\frac{\eta^{b,j}_{s-m+k/2-1}\xib_{s-m+k/2-j}}{\xia_{s-m+k/2-j}}r(q^s)\xia_sp_s(x)\\
&\qquad\qquad-\frac{\eta^{b,j}_{s+m+k/2-1}\xib_{s+m+k/2-j}}{\xia_{s+m+k/2-j}}I^{(1)}_{s}(r(q^s))\xia_{s-1}p_{s-1}(x)\Big]\dms\\
&=\int^{n}\Big[\ka^{b;k/2-m,j}_{s-m+k/2-1} r(q^s)\xib_sp_s(x)
-\ka^{b;k/2+m+1,j}_{s+m+k/2-1}I^{(1)}_{s}(r(q^s))\xib_{s-1}p_{s-1}(x)\Big]\dms,
\end{align*}
where in the last equality we used \eqref{5.2a}. In order to apply \leref{le5.4} with $\ell=b$ we need to check that 
$$I^{(1)}_{n}(\ka^{b;k/2-m,j}_{n-m+k/2-1})=\ka^{b;k/2+m+1,j}_{n+m+k/2-1},$$
which follows from equations \eqref{5.3} and \eqref{5.21}. Thus we can construct a $q$-difference operator $\cBh_z'$ such that the expression in \eqref{5.14a} can be rewritten as  
$$f^{(j)}_{n}\cBh_z'\xib_n p_n(x)=\cBh_z'\phi^{(j)}(\la_n)\xia_n p_n(x)
=\cBh_z' \phi^{(j)}(\cBz)\,\xia_n p_n(x),$$
completing the proof when $k$ is even. The case when $k$ is odd follows along the same lines, using the representation of $F_n$ described in \reref{re5.3} (iv).\qed

\begin{Remark}\label{re5.6}
Using \reref{re5.5}, we can write an explicit formula for the $q$-difference operator $\cBh_z^{h}$ in equation \eqref{3.10b} in terms of the operators $\cBz$, $\cBz^{\ell}$, $\ell\in\{a,b,c,d\}$. This might be useful for specific examples or small values of $k$, but the formula becomes very involved in general.
\end{Remark}

\section{Bispectral extensions which are orthogonal with respect to a measure on the real line} \label{se6}
In this section we assume (for simplicity) that the parameters $a,b,c,d$ are such that $\max(|a|,|b|,|c|,|d|)<1$. Then the Askey-Wilson polynomials $\{p_n(x;a,b,c,d)\}_{n=0}^{\infty}$ are mutually orthogonal with respect to the measure $\mu_{a,b,c,d}$ on  $[-1,1]$, defined by 
\begin{subequations}\label{6.1}
\begin{equation}\label{6.1a}
d\mu_{a,b,c,d}=\frac{w(x;a,b,c,d)}{2\pi \sqrt{1-x^2}}\,dx,
\end{equation}
with 
\begin{equation}\label{6.1b}
w(x;a,b,c,d)=\frac{(z^2,z^{-2};q)_{\infty}}{(az,a/z,bz,b/z,cz,c/z,dz,d/z;q)_{\infty}},
\end{equation}
see \cite[Theorem 2.2, p.~11]{AW}.
\end{subequations}

Next we explain how the techniques developed in the paper can be used to obtain extensions of the Askey-Wilson polynomials which eigenfunctions of higher-order $q$-difference operators and which are also mutually orthogonal with respect to a measure on $\Rset$. This  provides a new derivation of the constructions in our joint paper with L.~Haine \cite{HI}. Applying \thref{th3.1} to this situation, we obtain $q$-difference operators of lower-order than the ones constructed from the general theory in \cite{HI}.

The key point is to choose the polynomials $\phi^{(j)}(\la_{\ga})$ so that the space $W$ defined by \eqref{3.7} is $\cLga$ invariant (and therefore, the polynomials $\ph_n(x)$ will satisfy a three-term recurrence relation). The simplest possibility is to have $\psi^{(j)}(\la_{\ga})$ as eigenfunctions for the operator $\cL_{\ga}$, which in view of equation \eqref{4.1} is equivalent to pick the polynomial $\phi^{(j)}(\la_{\ga})$ to be an eigenfunction for the operator $\cL^{\de_j}_{\ga}$ for every $j$. Note that 
\prref{pr4.2} already gives us polynomial eigenfunctions for the operators $\cLga^{\ell}$ where $\ell\in\{a,b,c,d\}$ and therefore one natural choice is, for every $j=1,2,\dots,k$, to take 
\begin{equation}\label{6.2}
\phi^{(j)}(\la_{\ga})=g^{\de_j}_{k_j}(\la_{\ga}), 
\end{equation} 
where $k_j\in\Nset_0$. However, the corresponding extensions of the Askey-Wilson polynomials will depend only on the four parameters $(a,b,c,d)$ and will not contain as a limit the Krall-Jacobi or even the Krall-Laguerre polynomials. 

In order to get more interesting examples, we fix $\ell\in\{a,b,c,d\}$ and we try to find special values of the parameters $a,b,c,d$ for which the difference operator $\cLga^{\ell}$ has two dimensional polynomial eigenspaces (for specific eigenvalues). This would allow us to pick $\phi^{(j)}(\la_{\ga})$ in the corresponding eigenspaces (and they will depend on new free parameters). 

Recall that the Askey-Wilson polynomials $p_n(x)=p_n(x;a,b,c,d)$ are symmetric in the parameters $a,b,c,d$ and therefore we can rewrite formula \eqref{2.2} by exchanging the roles of $a$ and $d$, i.e. we have 
\begin{equation}\label{6.3}
p_n(x)=\frac{(ad,bd,cd;q)_n}{d^n}
\fpt{q^{-n}}{abcdq^{n-1}}{dz}{dz^{-1}}{ad}{bd}{cd}. 
\end{equation}
The ${}_{4}\phi_{3}$ series above clearly terminates for all $n\in\Cset$ if we fix $z=z_m=dq^{m}$ where $m\in\Nset_0$. Set $x_m=\frac{1}{2}(z_m+\frac{1}{z_m})$. Equations \eqref{2.5} and \eqref{4.1} show that 
\begin{equation}\label{6.4}
\xia_n p_n(x_m)=\frac{q^n}{(ad)^n}\frac{(abcd/q,ad;q)_n}{(bc,q;q)_n}
\fpt{q^{-n}}{abcdq^{n-1}}{dz_m}{q^{-m}}{ad}{bd}{cd}
\end{equation}
is an eigenfunction for the operator $\cL_n^{a}$ with eigenvalue $2x_m=z_m+\frac{1}{z_m}$. Moreover, since the ${}_4\phi_{3}$ series terminates, it is easy to see that the meromorphic extension of the expression in \eqref{6.4} is an eigenfunction for the bi-infinite extension $\cL_{\ga}^{a}$ (because after we cancel the exponents and the $q$-shifted factorials in front of the ${}_4\phi_{3}$ series, the eigenfunction equation can be rewritten as a polynomial equation in $q^{\ga}$, which if true for $\ga=n\in\Nset_0$, must be true for all $\ga\in\Cset$). Note also that, if $a=dq^{\al}$ where $\al\in\Nset$, the eigenvalue $z_m+\frac{1}{z_m}=dq^{m}+\frac{1}{dq^m}$ above, and $\frac{a}{q^{j+1}}+\frac{q^{j+1}}{a}=dq^{-j+\al-1}+\frac{1}{dq^{-j+\al-1}}$ in \eqref{4.5} coincide when $m=-j+\al-1$. Moreover, if $ad=q^{\al+l}$ where $l\in\Nset$ then the exponents and the $q$-shifted factorials multiplying the ${}_4\phi_{3}$ series in \eqref{6.4} can be rewritten as a Laurent polynomial in $q^{n}$ as follows
\begin{equation*}
\frac{q^n}{(ad)^n}\frac{(abcd/q,ad;q)_n}{(bc,q;q)_n}=\frac{1}{(bc,q;q)_{\al+l-1}}\,\frac{(bcq^{n},q^{n+1};q)_{\al+l-1}}{(q^{n})^{\al+l-1}}.
\end{equation*} 
It is not hard to see that the above expression is a polynomial of $\la_{n}$ (since it is invariant under the action of the involution $I^{(0)}_n$) and that the eigenfunctions in \eqref{4.4} and \eqref{6.4} are independent.

Summarizing the above comments, we obtain a different proof of Proposition 5.2 in \cite{HI}.

\begin{Proposition}\label{pr6.1}
Assume that $d=\var q^{l/2}$ and $a=dq^{\al}=\var q^{\al+l/2}$, where $\var=\pm1$, and $\al,l\in\Nset$. For $m=0,1,\dots,\al-1$, the polynomials 
\begin{subequations}\label{6.5}
\begin{align}
u^{1,m}(\la_{\ga})&=\fpt{q^{m-\al+1}}{q^{-m+\al+1}/a^2}{q^{\ga+1}}{1/abcdq^{\ga-2}}{q^2/ab}{q^2/ac}{q^2/ad},\label{6.5a}\\
u^{2,m}(\la_{\ga})&=\frac{(bcq^{\ga},q^{\ga+1};q)_{\al+l-1}}{(q^{\ga})^{\al+l-1}}
\fpt{q^{-\ga}}{abcdq^{\ga-1}}{q^{m+l}}{q^{-m}}{ad}{bd}{cd},\label{6.5b}
\end{align}
form a fundamental set of solutions for the eigenvalue problem
\begin{equation}\label{6.5c}
\cLga^{a}u_{\ga}=2x_m u_{\ga}, \quad\text{ where }\quad x_m=\frac{\var}{2}(q^{m+l/2}+q^{-m-l/2}).
\end{equation}
\end{subequations}
\end{Proposition}
The polynomials $u^{1,m}$, $u^{2,m}$ above correspond to $s_{\ga}(z_m)$ and $r_{\ga}(z_m)$ in \cite[page~302]{HI}. The connection between $u^{1,m}$ and $s_{\ga}(z_m)$ can be seen by applying Sears' transformation formula 
\cite[page~49, formula (2.10.4)]{GR}.

If the parameters $a$ and $d$ satisfy the conditions in \prref{pr6.1}, then we can pick $\de_{j}=a$ and 
$\phi^{(j)}(\la_{\ga})\in\Span\{u^{1,{m_j}}(\la_{\ga}),u^{2,{m_j}}(\la_{\ga})\}$, for some  $m_j\in\{0,1,\dots,\al-1\}$. Note that $\Span\{\phi^{(j)}(\la_{\ga})\}$ depends on one free parameter. If we make a similar choice for every $j=1,2,\dots,k$, then clearly $\cLga W\subset W$ (since each $\psi^{(j)}_{\ga}$ is an eigenfunction for $\cLga$). We shall assume that the polynomials $\phi^{(j)}(\la_{\ga})$ are generic so that the conditions \eqref{3.6} hold. Therefore $\fAz=\Cset[z+1/z]$ and the polynomials $\{\ph_n(x)\}_{n=0}^{\infty}$ will satisfy the three-term recurrence relation
$$\cLh_n\ph_n(x)=2x\ph_n(x), \quad \text{ where }\quad \cLh_n:=\cLh_n^{z+1/z}.$$

It is well known\footnote{See \cite[Proposition~6.1]{HI}, although this statement can be found already in \cite[pp.~14-19]{Wr} in the case of differential operators.} 
that the intertwining 
relation \eqref{4.7} means that the operator $\cLh_{\ga}$ can be obtained from $\cLga$ by a sequence of $k$ elementary Darboux transformations at the points $2x_{m_j}$. For the semi-infinite operators, this gives 
\begin{align}
&\cL_n=2x_{m_1}\Id+P_n^{(1)}Q_n^{(1)}\curvearrowright \cL^{(1)}_n=2x_{m_1}\Id+Q^{(1)}_nP^{(1)}_n=2x_{m_2}\Id+P^{(2)}_nQ^{(2)}_n
\curvearrowright\cdots                      \nonumber\\
&\quad \cL^{(k-1)}_n=2x_{m_{k-1}}\Id+Q^{(k-1)}_nP^{(k-1)}_n=2x_{m_k}\Id+P^{(k)}_nQ^{(k)}_n
                                \label{6.6}\\
&\qquad\curvearrowright \cLh_n=\cL^{(k)}_n=2x_{m_k}\Id+Q^{(k)}_nP^{(k)}_n,  \nonumber
\end{align}
where $P^{(j)}_n=\fp^{1,j}_nE_n+\fp^{0,j}_n\Id$ are first-order forward difference operators, $Q^{(j)}_n=\fq^{0,j}_n\Id+\fq^{-1,j}_nE_n^{-1}$ are first-order backward difference operators, and 
$$\cQ_n=Q_n^{(k)}Q_n^{(k-1)}\cdots Q_n^{(1)}.$$

If we define polynomials $p^{(j)}_n(x)=Q^{(j)}_n p^{(j-1)}_n(x)$, where $p^{(0)}_n(x)=p_n(x)$ are the Askey-Wilson polynomials, then $\ph_n(x)=p^{(k)}_n(x)$ are the polynomials constructed in \seref{se3}. Using \eqref{6.6}, it is not hard to prove inductively\footnote{See Theorem 2 in \cite{GHH}.} that the polynomials $\{p^{(j)}_n(x)\}_{n=0}^{\infty}$ are orthogonal with respect to a measure of the form
$$d\mu^{(j)}=\frac{d\mu_{a,b,c,d}}{\prod_{s=1}^{j}(x-x_{m_s})}+\sum_{s=1}^{j}\nu_s\de(x-x_{m_s}),$$
where $\mu_{a,b,c,d}$ is the Askey-Wilson measure in \eqref{6.1} and the constants $\nu_1,\dots,\nu_j$ are in one-to-one correspondence with the free parameters that specify the one dimensional subspace $\Span\{\phi^{(j)}(\la_{\ga})\}$ of $\Span\{u^{1,{m_j}}(\la_{\ga}),u^{2,{m_j}}(\la_{\ga})\}$ at each step. Putting all this together, we obtain the following theorem (c.f. Theorem 6.2 in \cite{HI}).

\begin{Theorem}\label{th6.2}
Suppose that $d=\var q^{l/2}$ and $a=\var q^{\al+l/2}$, where $\var=\pm1$, and $\al,l\in\Nset$ and let $k\leq \al$. 
\item[(I)]
For each $j=1,2,\dots, k$ pick $\de_j=a$ and 
\begin{equation}\label{6.7}
\phi^{(j)}(\la_{\ga})\in\Span\{u^{1,{m_j}}(\la_{\ga}),u^{2,{m_j}}(\la_{\ga})\}
\end{equation}
with $m_j\in\{0,1,\dots,\al-1\}$ and $m_j\neq m_i$ for $j\neq i$. Then for generic $\{\phi^{(j)}(\la_{\ga})\}$ as in \eqref{6.7}, the polynomials $\ph_n(x)$ defined by \eqref{3.1}-\eqref{3.2a} are orthogonal with respect to a complex measure of the form 
\begin{equation}\label{6.8}
d\hat{\mu}=\frac{d\mu_{\var q^{\al+l/2},b,c,\var q^{l/2}}}{\prod_{j=1}^{k}(x-x_{m_j})}+\sum_{j=1}^{k}\nu_j\de(x-x_{m_j})\text{ with }
x_{m_j}=\frac{\var}{2}(q^{m_j+l/2}+q^{-m_j-l/2}).
\end{equation}
\item[(II)] Conversely, if $m_j\in\{0,1,\dots,\al-1\}$ are distinct numbers, then for generic complex numbers $\nu_1,\dots,\nu_k$ we can use \eqref{3.1}-\eqref{3.2a} to construct polynomials $\{\ph_n(x)\}_{n=0}^{\infty}$ orthogonal with respect to the complex measure in \eqref{6.8} by picking $\de_j=a$ and an appropriate $\phi^{(j)}(\la_{\ga})$ as in \eqref{6.7} for every $ j=1,2,\dots,k$.
\end{Theorem}

We can use \thref{th3.1} to deduce that the polynomials $\ph_n(x)$, which are orthogonal with respect to the measure in \eqref{6.8}, are eigenfunctions of the $q$-difference operators in the algebra $\fDz$. In the above theorem, we can replace $a$ and $d$ with any two of the parameters $(a,b,c,d)$. In particular, if two of the parameters satisfy the above conditions with $\var=1$, and the other two satisfy analogous condition with $\var=-1$, we can add point masses on both sides outside of the interval $[-1,1]$. More precisely, suppose that 
$b=\var q^{\be+t/2}$, $c=\var q^{t/2}$, where $\be,t\in\Nset$ and $\var=\pm1$. Then for $m=0,1,\dots,\be-1$ we can define functions $v^{1,m}(\la_{\ga})$ and $v^{2,m}(\la_{\ga})$ by formulas \eqref{6.5a} and \eqref{6.5b}, respectively, by replacing $(a,b,c,d)$ with $(b,a,d,c)$, $\al$ with $\be$, and $l$ with $t$. This leads to the following theorem (cf. Corollary 6.6 in \cite{HI}).

\begin{Theorem}\label{th6.3}
Suppose that $a=q^{\al+l/2}$, $b= -q^{\be+t/2}$, $c=-q^{t/2}$, $d=q^{l/2}$, where $\al,\be,l,t\in\Nset$. Let $k_1,k_2\in\Nset$ be such that $k_1\leq \al$, $k_2\leq \be$ and set $k=k_1+k_2$. 
\item[(I)] For each $j=1,2,\dots, k_1$ pick $\de_j=a$ and 
\begin{equation}\label{6.9}
\phi^{(j)}(\la_{\ga})\in\Span\{u^{1,{m^1_j}}(\la_{\ga}),u^{2,{m^1_j}}(\la_{\ga})\}
\end{equation}
with $m^1_j\in\{0,1,\dots,\al-1\}$ and $m^1_j\neq m^1_i$ for $j\neq i$, and for each $j=k_1+1,\dots, k$ pick $\de_j=b$ and 
\begin{equation}\label{6.10}
\phi^{(j)}(\la_{\ga})\in\Span\{v^{1,{m^2_j}}(\la_{\ga}),v^{2,{m^2_j}}(\la_{\ga})\}
\end{equation}
with $m^2_j\in\{0,1,\dots,\be-1\}$ and $m^2_j\neq m^2_i$ for $j\neq i$. 
Then for generic $\{\phi^{(j)}(\la_{\ga})\}$ as in \eqref{6.9}-\eqref{6.10}, the polynomials $\ph_n(x)$ defined by \eqref{3.1}-\eqref{3.2a} are orthogonal with respect to a complex measure of the form 
\begin{equation}\label{6.11}
\begin{split}
d\hat{\mu}&=\frac{d\mu_{q^{\al+l/2},-q^{\be+t/2},-q^{t/2},q^{l/2}}}{\prod_{j=1}^{k_1}(x-x^{+}_{m^1_j})\prod_{j=1}^{k_2}(x-x^{-}_{m^2_j})}+\sum_{j=1}^{k_1}\nu^{1}_j\de(x-x^{+}_{m^{1}_j})+\sum_{j=1}^{k_2}\nu^2_j\de(x-x^{-}_{m^{2}_j}),\text{ with }\\
&\qquad x^{+}_{m^1_j}=\frac{1}{2}(q^{m^1_j+l/2}+q^{-m^1_j-l/2}) \text{ and }
x^{-}_{m^2_j}=-\frac{1}{2}(q^{m^2_j+t/2}+q^{-m^2_j-t/2}).
\end{split}
\end{equation}
\item[(II)] Conversely, suppose that $m^1_j\in\{0,1,\dots,\al-1\}$, $m^2_j\in\{0,1,\dots,\be-1\}$ and $m^{r}_j\neq m^{r}_i$ for $j\neq i$, $r=1,2$. Then for generic complex numbers $\nu^{1}_1,\dots,\nu^{1}_{k_1},\nu^{2}_1,\dots,\nu^{2}_{k_2}$ we can use \eqref{3.1}-\eqref{3.2a} to construct polynomials $\{\ph_n(x)\}_{n=0}^{\infty}$ orthogonal with respect to the complex measure in \eqref{6.11}
by picking
\begin{itemize}
\item $\de_j=a$ and an appropriate $\phi^{(j)}(\la_{\ga})$ as in \eqref{6.9} for every $ j=1,2,\dots,k_1$,
\item $\de_j=b$ and an appropriate $\phi^{(j)}(\la_{\ga})$ as in  \eqref{6.10} for every $ j=k_1+1,\dots,k$. 
\end{itemize}
\end{Theorem}

\begin{Remark}\label{re6.4}
As we noted earlier, we can use \eqref{6.2} to add Dirac measures without free parameters. We can mix this with the constructions in Theorems \ref{th6.2} and \ref{th6.3} to get even more complicated examples of orthogonal polynomials satisfying higher-order $q$-difference equations. 
\end{Remark}

\begin{Remark}\label{re6.5}
Recall that the substitution 
$$a=q^{\al+1/2}, \quad b=-q^{\be+1/2}, \quad c=-q^{1/2},\quad d=q^{1/2},$$
in the Askey-Wilson polynomials leads to $q$-extensions of the Jacobi polynomials, see \cite{Ra}. In particular, this explains how the polynomials $\{\ph_n(x)\}_{n=0}^{\infty}$ in Theorems \ref{th6.2} and \ref{th6.3} can be considered as extensions of the Krall-Jacobi polynomials.
\end{Remark}

\begin{Example}\label{ex6.6}
Consider the simplest application of \thref{th6.2} by taking $\ep=\al=l=1$, or equivalently $a=q^{3/2}$, $d=q^{1/2}$, $b,c$ arbitrary free parameters. Then $k=1$, $m_1=0$ and 
$$\phi^{(1)}(\la_{\ga})=u^{1,0}(\la_{\ga})+\bar{\nu}_1u^{2,0}(\la_{\ga})=1+\nu'_1\frac{(1-bcq^{\ga})(1-q^{\ga+1})}{q^{\ga}},$$
where $\nu'_1$ is a free parameter. Then $\tau_n=\phi^{(1)}(\la_{n})$ is a linear polynomial in $\la_n$ and therefore the algebra $\fAn$ will be generated by two polynomials of degrees $2$ and $3$, respectively. This means that the corresponding algebra $\fDz$ will be generated by two $q$-difference operators of orders $4$ and $6$, respectively. All this follows immediately from \thref{th3.1}, while the operator of minimal order that we can obtain from the general theory in \cite{HI} is of order $12$, see \cite[Section 7]{HI}. 

This example can be considered as a $q$-extension of the Krall-Jacobi polynomials \cite{Kr2}. The polynomials $\ph_n(x)$ are orthogonal with respect to a measure of the form
$$d\hat{\mu}=\frac{d\mu_{q^{3/2},b,c,q^{l/2}}}{(x-x_{0})}+\nu_1\de(x-x_{0})\text{ with }
x_{0}=\frac{1}{2}(q^{1/2}+q^{-1/2}).$$
Note that $x-x_0=-\frac{1}{2q^{1/2}}(1-zq^{1/2})(1-q^{1/2}/z)$, and therefore $\hat{\mu}$ can also be rewritten as
$$d\hat{\mu}=-2\sqrt{q}\,d\mu_{q^{1/2},b,c,q^{l/2}}+\nu_1\de(x-x_{0}).$$
The connection between $\nu'_{1}$ and $\nu_1$ can be derived from the orthogonality condition $\ph_1(x)\perp 1$.
\end{Example}

We conclude with several remarks concerning different constructions in the literature. 

\begin{Remark}
The first families of orthogonal polynomials which are eigenfunctions of higher-order $q$-difference operators, extending the little $q$-Jacobi and Laguerre polynomials were obtained already in \cite{GH3,VZ}.
\end{Remark}

\begin{Remark}\label{re6.8}
The results in \cite{HI} are formulated for a meromorphic (in $n$) extension of the polynomials $\ph_n(x)$.  The main difference with the formulation here is that in equation \eqref{3.2a} we need to replace $p_n(x)$ by a meromorphic extension of the Askey-Wilson polynomials (which can be written in terms of ${}_8\phi_7$ basic hypergeometric series \cite{IR,KSt,S} or can be defined recursively \cite{GH2}). The fact that the extended functions $\ph_n(x)$ satisfy bispectral equations can be deduced from \thref{th3.1} by analytic continuation. For details, we refer the reader to \cite{GI} where similar arguments were used in a multivariable setting.
\end{Remark}

\begin{Remark}
In the recent preprint \cite{D}, methods similar to the ones in \cite{I2,I3} were used to construct bispectral polynomials in the simplest case $k=1$ extending the discrete classical orthogonal polynomials (Charlier, Meixner, Krawtchouk and Hahn). When $k=1$, the right-hand side of equation \eqref{3.2a} is a $2\times 2$ determinant and can simply be written as a linear combination of $p_n(x)$ and $p_{n-1}(x)$. Most of the computations in this case can be done directly bypassing \leref{le5.2}, \reref{re5.3} and \leref{le5.4}. The corresponding polynomials can be obtained from the ones considered here by taking $k=1$ and by reducing the Askey-Wilson polynomials to the corresponding family of discrete classical polynomials.
\end{Remark}

\begin{Remark}
Note that the discrete parts of the measures in Theorems \ref{th6.2} and \ref{th6.3} are supported at specific points of the form $\pm\frac{1}{2}(q^{s/2}+q^{-s/2})$, $s\in \Nset$,
outside of the interval $[-1,1]$. They accumulate to the boundary 
$\pm 1$ only in the limit $q\rightarrow 1$. In the recent paper \cite{AS}, the authors consider different extensions of the Askey-Wilson polynomials, not related to the bispectral problem, by adding mass points at the end points $\pm 1$ of the interval $[-1,1]$.
\end{Remark}

\end{document}